\documentclass[a4paper]{amsart}

\usepackage{amsmath, amsfonts, amsthm, amssymb, mathrsfs, fullpage, enumerate, verbatim,color,float}

\newcommand{\N}{\mathbb{N}}
\newcommand{\T}{\mathcal{T}_n}
\newcommand{\C}{\mathfrak{C}_n}
\newcommand{\CS}{\mathfrak{C}}

\newcommand{\supp}{\text{supp}}

\usepackage{tikz}
\usetikzlibrary{arrows}

\newtheorem{thm}{Theorem}
\newtheorem{lem}[thm]{Lemma}

\newtheorem{cor}[thm]{Corollary}
\newtheorem{rem}[thm]{Remark}

\let\oldmarginpar\marginpar
\renewcommand\marginpar[1]{\-\oldmarginpar[\raggedleft\footnotesize #1]%
{\raggedright\footnotesize #1}}

\title{Some isomorphism results for Thompson like groups $V_n(G)$}
\author{Collin Bleak, Casey Donoven and Julius Jonu\v{s}as}

\begin{document}
  \maketitle
  
\begin{abstract}
In this paper, we provide some isomorphism results for the groups $\{V_n(G)\}$, 
supergroups of the Higman-Thompson group $V_n$
where $n\in\N$ and $G\leq S_n$, the symmetric group on $n$ points.
These groups, introduced by Farley and Hughes, are the groups generated by $V_n$ and the tree automorphisms $[\alpha]_g$
defined as follows. 
For each $g\in G$ and each node $\alpha$ in the infinite rooted $n$-ary tree, 
the automorphisms $[\alpha]_g$ acts iteratively as $g$ on the child leaves of $\alpha$ and
every descendent of $\alpha$. In particular, we show that $V_n\cong V_n(G)$ if and only
if $G$ is semiregular (acts freely on $n$ points) and some additional sufficient conditions
for isomorphisms.  Essential tools in the above work are a study of the dynamics of the 
action of elements of $V_n(G)$ on the Cantor space, Rubin's Theorem, and transducers from Grigorchuk, 
Nekrashevych, and Suschanski\u{i}'s rational group on the $n$-ary alphabet.  
\end{abstract}

\section{Introduction}

    In this paper, we consider groups $\{V_n(G)\}$, where, for each $n\in\N$ and 
each permutation group $G\leq S_n$, 
we have $V_n(G)$ as a supergroup of the Higman-Thompson group $V_n$.    
(We follow Brown in \cite{BrownFinitenessProps} using $V_n$ to denote the group $G_{n,1}$ of 
Higman's book \cite{HigmanFPSG} and $V_2=V$.)
Specifically, we define the groups $\{V_n(G)\}$ as below.
    
      Consider the infinite regular $n$-ary rooted tree $\T$, with vertices the set of words 
in the free monoid $\{1, 2, ..., n\}^*$ and with an edge labelled $j$ between all pairs $(w,wj)$ 
of vertices where $w\in \{1,2,\ldots ,n\}^*$ and $j\in\{1,2,\ldots,n\}$.  We set $\CS_n$ to be 
the Cantor space $\{1,2,\ldots,n\}^\omega$ which is the boundary of this infinite tree,  points 
of which correspond to infinite geodesic paths from the root.  For a vertex $\theta$ in the tree 
we define the \emph{cone} of $\theta$ to be the set of all points from $\CS_n$ which have 
$\theta$ as a prefix, i.e.
\[
  [\theta] = \{ \theta u :  u \in \{1, 2, ..., n\}^{\omega} \}\subset \C.
\] 
As is well known, the set of all such cones is a clopen basis for the topology on $\CS_n$.   
Now, let $[\alpha_1], [\alpha_2], ..., [\alpha_t]$ and $[\beta_1], [\beta_2], ..., [\beta_t]$ be two
partitions of $\CS_n$ and define the map $v$ from the $\CS_n$  to itself by
$(\alpha_i\gamma)\cdot v=\beta_i\gamma$ for
all $i$ and $\gamma\in\{1, 2, ..., n\}^{\omega}$.  The Higman Thompson group $V_n$ is the 
group of all such
mappings.  In the above situation, we might say that $v$ is a \emph{prefix substitution map},
taking the set of prefixes $\{\alpha_i\}$ to the set of prefixes $\{\beta_i\}$.  

Now suppose 
$G\leq S_n$ and $g\in G$ and consider $[\alpha]_g$, the homeomorphism of $\CS_n$ which is 
the identity outside of the cone $[\alpha]$ and which acts on the cone $[\alpha]$ as 
$$\alpha\beta\cdot[\alpha]_g=\alpha\|(b_1\cdot g)\|(b_2\cdot g)\|(b_3\cdot g)...$$ for all 
$\beta\in\CS_n$.  (We will use $w\| j$ to concatenate the word $w$ with the word/letter $j$ 
when it seems to us that this notation might help the reader)  We now define $V_n(G)$ as follows:
\[
V_n(G)=\langle V_n, [\alpha]_g : g\in G, \alpha\in\{1, 2, ..., n\}^*\rangle.
\]
We call $[\alpha]_g$ an \emph{iterated permutation}.

This article investigates the isomorphism classes of the groups $V_n(G)$, with some perhaps unexpected results.  In particular, $V_n(G)\cong V_n$ if and only if $G$ 
is semiregular. (Recall $G\leq S_n$ is semiregular if it acts freely on $\{1,2,\ldots,n\}$.)

The groups $\{V_n(G)\}$ sit naturally in two families of constructed groups which have arisen 
in previous research. Elizabeth Scott in her research \cite{Scott84I, Scott84II, Scott84III} 
describes the first such family of groups which are developed further by Claas 
R\"over, specifically including an extension of $V$ with the Grigorchuk group $\Gamma$.
The second family are the finite similarity structure groups of Hughes (FSS groups), 
which are the focus of study in \cite{HughesFSS, FarleyHughesFSS}.

The class of groups $\{V_n(G)\}$ is first explicitly studied in 
\cite{FarleyHughesFSS}.  Farley and Hughes show that if $G$ is not semiregular, then $V_n(G)$ 
is not isomorphic to $V_n$ and that the commutator subgroup $V_n(G)'$ is a finite 
index simple subgroup of $V_n(G)$. The proof of the first result is to use Rubin's theorem 
to translate the question into topological dynamics, and to show that the respective groups 
of germs are not the same (this follows the same outline as the argument of Bleak and Lanoue 
in \cite{BleakLanoue} where it is shown that if $m\neq n$, then Brin's higher-dimensional 
Thompson groups $mV$ and $nV$ are not isomorphic).  This approach fails to distinguish $V_n$ from $V_n(G)$ when $G$ is semiregular, as in this case, the respective groups of germs are isomorphic. 

Of late, the structure of groups built as supergroups of the 
Higman-Thompson groups $V_n$ has become a topical focus of interest.  We can mention the papers of  
\cite{RoverGrig,NekrashevychVGroups,FarleyHughesFSS,BelkMatucciFInf,Thumann} which all 
explore properties of such groups.  Our own perspective has been heavily influenced by the 
dynamical methods which have arisen through the use of Rubin's theorem \cite{McClearyRubin, Rubin} 
and, as will be seen below, through the interaction between the theory of the extended 
Thompson type groups and the theory of groups of automata (c.f., \cite{GNS}).  Indeed, 
subgroups of Grigorchuk, Nekrashevych, and Sushchanski\u\i's Rational group seem to play 
an ever-increasing part in our research and in related work to the extended Thompson type 
groups (see \cite{GNS, BelkBleak,AutVn,NekrashevychVGroups}).

\subsection{Some further notation}
 Throughout this paper we will be composing functions from left to right.  We will also 
label words in the alphabet $(1,2,\ldots,n)$ with Greek letters and single letters with 
the Latin alphabet to make the distinction clearer, while using $\|$ to indicate 
concatenation when it is not obvious, as mentioned before.

One way to characterise homeomorphisms of Cantor space is via certain automata which are called
\emph{transducers}.  Finite transducers have a long history, with the 
abstract theory stretching back to Huffman \cite{Huffman}, and the concrete theory 
stemming from modelling ``Sequential machines'' such as cash registers and etc..  
Gl\v{u}skov in \cite{Gluskov63} gives a thorough introduction both to the abstract 
theory of these automata and to the history of the development of that theory.  We will be heavily influenced by the point of view of Grigorchuk, Nekrashevych, and Sushchanski\u\i\;\!\! in \cite{GNS} where they formally introduce the \emph{rational group};  the full group of homeomorphisms of Cantor space which can be described by finite transducers as described in more detail below.

 A \emph{transducer} 
$T_{q_0}$ is defined as a sextuple $T_{q_0}=\langle X_I, X_O, Q, \pi, \lambda, q_0 \rangle$ where 
$X_I$ and $X_O$ are finite alphabets, $Q$ is a set of states, $\pi:X_1\times Q\to Q$
and $\lambda:X_I\times Q\to X_O^*$ are mappings, and $q_0\in Q$. A transducer acts on finite and infinite
words over the input alphabet $X_I$ in a recursively defined fashion.
Let $\alpha=a_1a_2a_3...$ be a word.  Define 
\[
\alpha\cdot T_{q_0}=(a_1,q_0)\lambda\|\big((a_2a_3...)\cdot T_{(a_1,q_0)\pi}\big),
\]
where $T_{(a_1,q_0)\pi}$ is the same transducer as $T_{q_0}$, except with a different start state.
This shows how $\lambda$ can be thought of as being a rewrite rule and $\pi$ as a transition function,
determining which rewrite 
rule to use. A transducer is \emph{synchronous} if it preserves the length of words it acts on, i.e.
$\lambda:X_I\times Q\to X_0$.

In an abuse of notation, we also give the transition function $\pi$ words for input 
instead of just single letter inputs, such as $(\rho,q_0)\pi$.  This is simply shorthand for the last state
the transducer would enter if had acted on the whole word one letter at a time.  For example,
\[(\rho\|\chi)\cdot T_{q_0}=(\rho\cdot T_{q_0})\|(\chi\cdot T_{(\rho,q_0)\pi}).\]  

Let $X$ be a finite set with cardinality at least two.  The paper \cite{GNS} describes how any homeomorphism of the Cantor space $X^\omega$ can be represented by a transducer   $T_{q_0}=\langle X, X, Q, \pi, \lambda, q_0 \rangle$ with $Q$ an infinite set.  However, if one restricts to the homeomorphisms which can be represented by such transducers with $Q$ finite, then one obtains a group, which Grigorchuk et al. \!refer to as the \emph{rational group over alphabet $X$}.  These homeomorphisms are the ones which admit precisely finitely many types of ``local actions.''

Following \cite{GNS}, which shows that elements of $V_2$ can be characterised by transducers representing elements of the rational group on a two letter alphabet, the elements of $V_n$ can also be characterised by such finite 
transducers (on an $n$ letter alphabet).  These transducers admit an even stronger condition; for any such element $v$ there is a finite 
transducer $\tau_v$ representing the homeomorphism, and an $m\in \N$ so that all paths 
of length $m$ from the initial state result in an ``identity state'' after which no 
changes occur to the infinite string characterising the initial point in Cantor space 
which is being transformed.

Some elements of $V_n(G)$ simply permute finite prefixes, meaning the partitions from the 
definition of elements of $V_n$ are equal. For convenience, we will introduce some notation for 
these elements. We will write $(\alpha_1\;\alpha_2\;\ldots\;\alpha_t)$ to denote the element that maps
$\alpha_i\gamma$ to $\alpha_{i+1}\gamma$ for all $\gamma\in\{1,2,...,n\}^{\omega}$ with 
$\alpha_{t}\gamma$ mapping to $\alpha_1\gamma$. In the case where an element $v\in V_n$
exactly permutes a subset of the children of a prefix $\alpha$, i.e. a subset of
$\{\alpha1,\alpha2,...,\alpha n\}$, we write $v=\lfloor\alpha\rfloor_h$ where $h\in S_n$ maps $i$ to $j$ if
and only if $v$ maps $[\alpha\| i]$ to $[\alpha\| j]$.

 Note that we will often use the word transducer 
when we actually mean the homeomorphism of Cantor space which arises by applying that 
transducer to all the infinite 
strings which represent the points of Cantor space. 

\subsection{Statement of Results}
In this section, we record our two main results on finding isomorphisms between group in the family $\{V_n(G)\}$.
 
Recall that a group $G\leq S_n$ is \emph{semiregular} if $g$ has no fixed points for all non-trivial $g\in G$, 
i.e. $\text{Stab}_G(x)=\{id\}$ for all $x\in\{1,2,...n\}$.

Our first theorem is perhaps surprising in that many experts had the opinion that there would be many isomorphism 
types for groups $V_n(G),$ with $G$ a semiregular subgroup of $S_n$ (e.g., see \cite{FarleyHughesFSS},
 where such an expectation is expressed at the end of Section 7, in the discussion following Example 7.24).

\begin{thm}\label{semi}
  Let $n\geq 2$ and $G \leq S_n$. Then $V_n(G) \cong V_n$ if and only if $G$ is semiregular. 
\end{thm}

In fact the construction of the maps we build to take $V_n(G)$ to $V_n$ when $G$ is 
semiregular is fairly general, and we can follow the idea of the construction even 
when $G$ is not semiregular.  From that perspective we have that the  previous 
theorem can be seen as a sub-case of the following result.
(Note that below we use $\text{Stab}_{S_n}(R)$ to indicate the setwise stabilizer 
of $R$ in $S_n$, not the pointwise stabilizer.)

\begin{thm}\label{thm:fullthm}
Let $n\geq 2$ and let $H\leq S_n$ be semiregular. Also let $R$ be a set of orbit 
representatives of $H$'s natural action on $\{1,2,...,n\}$,
and let $G\leq N_{S_n}(H)\cap \text{Stab}_{S_n}(R)$. Then $HG$ is a group and $V_n(HG)\cong V_n(G)$ 
\end{thm}

The afore mentioned transducers play an essential role in building the isomorphisms above.

\section{Semiregular Groups}

In this section, we prove the forward implication in Theorem \ref{semi}, showing that semiregularity is 
a necessary condition for an isomorphism to exist. This follows directly from Farley and Hughes'
non-isomorphism result in \cite{FarleyHughesFSS}.  We provide a similar proof, tailored specifically for
the case of semiregular groups.

\subsection{Rubin's Theorem}

Farley and Hughes use Rubin's Theorem \cite{Rubin} to conclude 
that isomorphisms between $V_n(G)$ and $V_m(H)$ must be via conjugation, which allows them to
make restrictions on isomorphisms by examining the dynamical structure of elements in the groups.

In order to state Rubin's Theorem, we need to give the definition of a group $G$ acting in a 
\emph{locally dense} way on a set $X$.  If $X$ is a topological space and $G$ is group of 
homeomorphisms from $X$ to $X$, $G$ is locally dense if and only if for any $x\in X$ and open 
neighbourhood $U$ of $x$, the set $\{x\cdot g\;|\;g\in G,g|_{X \setminus U}=1|_{X\setminus U}\}$ has 
closure containing an open set.\vspace{\baselineskip}

\noindent {\bf Rubin's Theorem }{\it
Let $X$ and $Y$ be locally compact, Haudorff topological spaces without isolated points, let $A(X)
$ and $A(Y)$ be the automorphism groups of $X$ and $Y$, respectively, and let $G\leq A(X)$ and 
$H\leq A(Y)$.  If $G$ and $H$ are isomorphic and are both locally dense, then for each 
isomorphism $\phi: G\to H$, there is a unique homeomorphism, $\psi:X \to Y$, so that for each $g
\in G$, we have $(g)\phi=\psi^{-1}g\psi.$}\vspace{\baselineskip}

We will refer to the conjugating homeomorphism $\psi$ (which realises the isomorphism $\phi$) as the \emph{Rubin conjugator}.

\begin{lem}
Let $G\leq S_n$ and $H\leq S_m$.
If $\phi:V_n(G)\to V_m(H)$ is an isomorphism, then there exists a unique homeomorphism, $\psi$,
from the $n$-ary Cantor set to the $m$-ary Cantor set such that for every $v\in V_n(G)$ we have
$(v)\phi=\psi^{-1}v\psi$.
\end{lem}

\begin{proof}
It is known that Cantor sets satisfy all of the conditions in Rubin's Theorem and it is easy to see that  
$V_n$ acts locally densely on the $n$-ary Cantor set for all integer $n\geq2$.  Since $V_n\leq V_n(G)$, $V_n(G)$ is also locally dense
and the lemma follows directly from Rubin's Theorem.
\end{proof}

{The next lemma  discusses an easy isomorphism which arises as a result of topological conjugacy.}

\begin{lem}
{Let $G$ and $H$ be conjugate subgroups of the symmetric group. Then $V_n(G)\cong V_n(H)$.}
\end{lem}

The isomorphism between the two is simply given by a re-labeling of $\T$ using an iterated version of the conjugation from
$G$ to $H$ in $S_n$. This greatly reduces the number of cases one needs to consider
when solving the isomorphism problem in this family of groups.

\subsection{Semiregularity} We now examine the dynamics near fixed points of specific elements in $V_n(G)$ to show how
the orbit structure can prevent isomorphisms from arising via Rubin's theorem.


\begin{lem}\label{lem:anticonjugacy}
Let $g\in G\leq S_n$ and $x\in\{1,...,n\}$ be such that $x\cdot g=x$ but $g\neq Id$.
The element $[\emptyset]_{g}\in V_n(G)$ is not conjugate to any element in $V_n$.
\end{lem}

\begin{proof}
The point $\gamma=x^\omega\in\{0,...,n\}^\omega$ is fixed under $[\emptyset]_{g}$.  Since $g\neq Id$,
there exists $y\in\{1,...,n\}$ such that $y\cdot g\neq y$.  Let $\gamma_k=x^{k-1}yx^\omega\in\{1,...,n\}^\omega$
with a $y$ in the $k$th position.  The points $\gamma_k$ have finite nontrivial orbit under
$[\emptyset]_{g}$ and can be found in every open cone containing $\gamma$

In order for an element $v\in V_n$ to have a fixed point, $v$ must take a cone $[\alpha]$
to another cone $[\beta]$ that has nontrivial intersection with $[\alpha]$.

Case 1: $[\alpha]=[\beta]$  This implies that every point within the cone $[\alpha]$ is a fixed point.
None of these fixed points will have points of nontrivial orbit within small open neighbourhoods.

Case 2: $[\beta]\subseteq[\alpha]$  This implies there exists a word $\gamma$ such that  
$\beta=\alpha\|\gamma$.  Let $\chi$ be an infinite word and consider that
$(\alpha\|\chi)\cdot v^m=\alpha\|\gamma^m\|\chi$ for all $m$.  This shows that the only fixed point of $v$ is
$\alpha\|\gamma^\omega$ and all other points in $[\alpha]$ have infinite orbit.

Case 3: $[\alpha]\subseteq [\beta]$ The element $v^{-1}$ takes cone $[\beta]$
to $[\alpha]$.  This was discussed is Case 2 and therefore $[\beta]$ contains one fixed point and the rest
have infinite orbit under $v^{-1}$.  An element and its inverse have the same orbit structure so $[\alpha]$
also contains a fixed point and points of infinite orbit.

Conjugation preserves orbit structure and since no elements of $V_n$ have nontrivial finite orbits arbitrarily close
to a fixed point, $[\emptyset]_{g}$ can not be conjugate to any element of $V_n$.
\end{proof}

We can now state the following corollary.

\begin{cor}
  Let $G \leq S_n$. If $V_n(G) \cong V_n$, then $G$ is semiregular. \label{cor:semiregularnecessity}
\end{cor}


In the case when $G$ is semiregular, we can build the conjugating homeomorphism from Rubin's Theorem
using the transducers which we will see in the next section.  Henceforth, $H$ will represent a semiregular
subgroup of $S_n$ and $R=\{x_1,...,x_k\}$ 
will be an orbit transversal of $H$, a set of
orbit representatives of the natural action of $H$ on $\{1,2,...,n\}$.

Note that there is a unique
element $h_i\in H$ such that $i\cdot h_i\in R$.  Suppose that there exists a second element $h_i'$ such that 
$i\cdot h_i'\in R$. Then $i\cdot h_i=i\cdot h_i'$ since the orbit representative of $i$ is unique and
therefore $i\cdot h_i'h_i^{-i}=i$.  The only element in $H$ with a fixed point is the identity, meaning $h_i=h_i'$.

The following lemma demonstrates a useful
relationship between $N_{S_n}(H)\cap\text{Stab}_{S_n}(R)$ and $H$, alluded to in 
Theorem \ref{thm:fullthm}. 
Indeed, when $H$ is regular, i.e. both semiregular and transitive, $N_{S_n}(H)\cap\text{Stab}_{S_n}(R)\cong Aut(H)$.  The
peculiar form of multiplication in the following lemma is not unlike twisted conjugacy (particularly when $h=id$) and arises
through conjugation with transducers.  

\begin{lem}\label{lem:semiconjugacy}
Let $G\leq N_{S_n}(H)\cap\text{Stab}_{S_n}(R)$. Then the group $\langle H,G\rangle$ is equal to $HG$ and
for all $h\in H$, $g\in G$, and $x\in X$, \[h_x^{-1}hgh_{x\cdot hg}=g.\]
\end{lem}
\begin{proof}
Let $G\leq N_{S_n}(H)\cap\text{Stab}_{S_n}(R)$. This implies $\langle H,G\rangle\leq N_{S_n}(H)$ but $G\cap H=\{id\}$
since only the identity of $H$ can stabilize a point.  Among other things, this shows that $\langle H,G\rangle=HG$.

Next, let $h\in H$, $x\in X$, and $r\in R$ be the orbit representative of $x$. Consider the action of $h_x^{-1}hh_{x\cdot h}$ on $r$.
\[
r\cdot h_x^{-1}hh_{x\cdot h} = (x\cdot h)\cdot h_{x\cdot h}\in R 
\]
Since $r$ is the representative of its own orbit and
$h_x^{-1}hh_{x\cdot h}$ maps $r$ to itself, the element $h_x^{-1}hh_{x\cdot h}\in H$ must be the identity.

Let $g\in G$ and $x\in X$. Because $g\in N_{S_n}(H)$, conjugating $h_x$ by $g$ gives $g^{-1}h_xg=h_y$ 
for some $y\in X$.
Let $r\in R$ be the orbit representative of $x$ and consider the action of $h_y$ on $x\cdot g$, we have
\begin{eqnarray*}
(x\cdot g)\cdot h_y&=&(x\cdot g)\cdot g^{-1}h_xg\\
&=&x\cdot h_x g\\
&=& r\cdot g \in R 
\end{eqnarray*}
since $g\in \text{Stab}_{S_n}(R)$.  This shows that $h_y$ maps $x\cdot g$ into $R$, i.e. $h_y=h_{x\cdot g}$. Therefore
\[h_x^{-1}gh_{x\cdot g}=g.\]
Putting these together gives
\begin{eqnarray*}
h_x^{-1}\;hg\;h_{x\cdot hg}&=&(h_x^{-1}\;h\;h_{x\cdot h})(h_{x\cdot h}^{-1}\;g\;h_{(x\cdot h)\cdot g})\\
&=&g.
\end{eqnarray*}
\end{proof}

\section{Rubin conjugators as transducers}
In this section we construct transducers which generate the Rubin conjugators realising isomorphisms between $V_n(GH)$ and $V_n(G)$ for $H$ semiregular and appropriate $G$.  They are built from the semiregular group $H$
that we have added to $V_n$ and a choice $R$ of orbit transversal for the action of $H$.   We also give some specific examples and make several general calculations to describe the Rubin conjugator's interaction with elements of $V_n(GH)$.   The isomorphism of Theorem \ref{thm:fullthm} is proven in Section \ref{sec:Isomorphism}.

Remembering that $H$ is semiregular, $R$ is an orbit transversal of $H$, and letting $h\in H$, 
define the synchronous transducer $A_{H,R,h}$ as
\begin{eqnarray}
A_{H,R,h}&=&\langle\{1,2,...,n\},\{1,2,...n\},H,\pi,\lambda, h \rangle. \label{def:transducer}
\end{eqnarray}
The rewrite function $\lambda$ and transition function $\pi$ are defined
for all $i\in X=\{1,2,...,n\}$ and $g\in H$ as 
\[(i,g)\lambda=i\cdot g\quad\text{and}\quad(i,g)\pi=h_i,\]
where $h_i$ is the unique
element such that $i\cdot h_i\in R$. 
Since $H$ and $R$ are understood, we simplify the notation so that $A_{H,R,h}=A_h$.  

The inverse of the transducer, $A_{h}^{-1}$, can be described as
\[A_{h}^{-1}=\langle\{1,2,...,n\},\{1,2,...n\},H,\pi',\lambda', h \rangle.\]
The `inverse' rewrite function $\lambda'$ and `inverse' transition function $\pi'$ are defined
for all $i\in X\{1,2,...,n\}$ and $g\in H$ as 
\[(i,g)\lambda'=i\cdot g^{-1}\quad\text{and}\quad(i,g)\pi'=h_{i\cdot g^{-1}}.\]
It is useful (and interesting) to note that $\pi$ does not depend
on the current state of the
transducer, i.e. $(i,g)\pi=(i,h)\pi$ for all $i\in X$ and $g,h\in H$, whereas $\pi'$
does depend on the state.  One way to approach this idea is to see that $(i,h)\pi$ depends on
the input letter $i$ but $(i,h)\pi'$ depends on the output letter of $\lambda'$,
$i\cdot h^{-1}$.

{\flushleft{\it An example conjugating $V_2(S_2)$ to $V_2$}}:\\
As an example, consider the group $V_2(S_2)$ where $S_2$ is the two element semiregular permutation group 
$S_2=\langle(1\;2)\rangle$ acting on $\{1,2\}$ with orbit transversal $R=\{1\}$.  We build the explicit transducer giving the Rubin conjugator that takes $V_2(S_2)$ to $V_2$.

This transducer will have two states, $id$ and $(1\;2)$, and both states act on input letters as the permutation for which they are named. An input
of 1 will send the transducer to the state $id$, since $1\cdot id\in R$, and 2 as an input
will send it to the state $(1 \;2)$.  This can be represented pictorially in Figure \ref{fig:twostate}, with labeled 
circles as states, the arrows representing transitions between states, and the labels on arrow
giving the input on the left and output on the right.

\begin{figure}[H]
\centering
\begin{tikzpicture}[->,auto]
                 \node at (0,0) [circle, draw,minimum size=1cm] (A) {id};
                 \node at (2,0) [circle, draw,minimum size=1cm] (B) {(12)};
                 \draw[->] (A) to [out=45,in=135]      node {2/2} (B);
                 \draw[->] (B) to [out=225,in=315]     node {1/2} (A);
                 \draw[->] (A) to [loop left]          node {1/1} (A);
                 \draw[->] (B) to [loop right]         node {2/1} (B);
\end{tikzpicture}
\caption{The transducer for $H=\langle(1\;2)\rangle\leq S_2$ with $R=\{1\}$ and the start state unspecified.}\label{fig:twostate}
\end{figure}
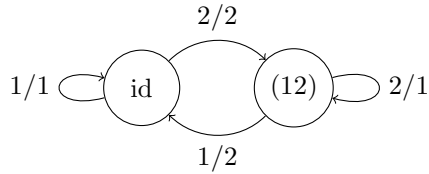

The inverse of the last transducer can be seen in Figure \ref{fig:twostateinv}.  Note that state you are sent to depends on the
output of the rewrite rule and not the input.

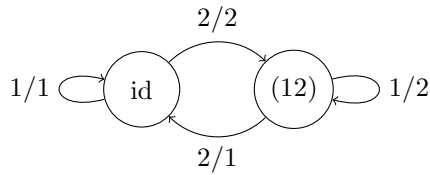
\begin{figure}[H]
\centering
\begin{tikzpicture}[->,auto]
                 \node at (0,0) [circle, draw,minimum size=1cm] (A) {id};
                 \node at (2,0) [circle, draw,minimum size=1cm] (B) {(12)};
                 \draw[->] (A) to [out=45,in=135]      node {2/2} (B);
                 \draw[->] (B) to [out=225,in=315]     node {2/1} (A);
                 \draw[->] (A) to [loop left]          node {1/1} (A);
                 \draw[->] (B) to [loop right]         node {1/2} (B);
\end{tikzpicture}
\caption{The inverse of the transducer for $H=\langle(1\;2)\rangle\leq S_2$ with $R=\{1\}$ and the start state unspecified.}\label{fig:twostateinv}
\end{figure}

{\flushleft{\it An example conjugating $V_3(G\;\!C_3)$ to $V_3(G)$}}:\\
Figure \ref{fig:triangular} depicts another transducer which is constructed to produce a homeomorphism which will conjugate the group $V_3(G\;\!C_3)$ to $V_3(G)$, where the group $C_3$ is specifically the semiregular cyclic group of order three in $S_3$; $C_3=\langle (1\;2\;3)\rangle$.  To build our transducer we choose  $R=\{1\}$ for our orbit transversal.  Note that appropriate $G$ must stabilise $R=\{1\}$ and normalise the (already normal) subgroup $C_3$.

\begin{figure}[H]
\centering
\begin{tikzpicture}[->,auto]
                 \node at (0,0) [circle, draw,minimum size=1cm] (A) {id};
                 \node at (2,3) [circle, draw,minimum size=1cm] (B) {(132)};
                 \node at (4,0) [circle, draw,minimum size=1cm] (C) {(123)};
                 \draw[->] (A) to [out=90,in=180]     node {2/2} (B);
                 \draw[->] (B) to [out=210,in=60]     node {1/3} (A);
                 \draw[->] (A) to [out=15,in=165]     node {3/3} (C);
                 \draw[->] (C) to [out=195,in=345]    node {1/2} (A);
                 \draw[->] (B) to [out=0,in=90]       node {3/2} (C);
                 \draw[->] (C) to [out=120,in=330]    node {2/3} (B);
                 \draw[->] (A) to [loop left]         node {1/1} (A);
                 \draw[->] (B) to [loop above]        node {2/1} (B);
                 \draw[->] (C) to [loop right]        node {3/1} (C);
\end{tikzpicture}
\caption{The transducer for $C_3=\langle(1\;2\;3)\rangle\leq S_3$ with $R=\{1\}$ and the start state unspecified.}\label{fig:triangular}
\end{figure}
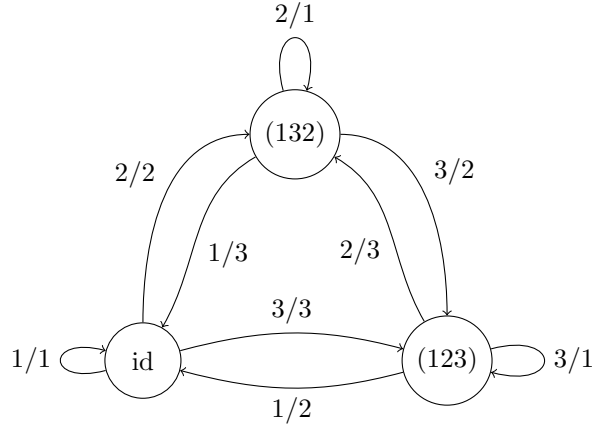

To see that the transducers $A_{H,R,id}$ constructed along the lines above for specific semiregular $H$ and choice of $R$ produce the advertised isomorphisms between $V_n(GH)$ and $V_n(G)$ for semiregular $H$ and appropriate $G$, we look at where the generators of $V_n(G)$ and $V_n(GH)$ are taken under conjugation. The following lemmas are calculations to assist in computing conjugation by the homeomorphisms produced by such transducers.


\begin{lem}
Let $h,g\in H$. Then $A_h=A_g\lfloor\emptyset\rfloor_{g^{-1}h}$.
\end{lem}

\begin{proof}
Let $\chi=x_1x_2x_3\ldots\in X^{\omega}$, let $g,h\in H$ and consider $\chi\cdot A_g\lfloor\emptyset\rfloor_{g^{-1}h}$,
\begin{eqnarray*}
\chi\cdot A_g\lfloor\emptyset\rfloor_{g^{-1}h}&=&(x_1\cdot g\|(x_2x_3...)\cdot A_{(x_1,g)\pi})\cdot\lfloor\emptyset\rfloor_{g^{-1}h}\\
&=&(x_1\cdot gg^{-1}h)\|\big((x_2x_3...)\cdot A_{(x_1,g)\pi}\big)\\
&=&(x_1\cdot h)\|\big((x_2x_3...)\cdot A_{(x_1,h)\pi}\big)\\
&=&\chi\cdot A_{h}.
\end{eqnarray*}
since $(x_1,g)\pi=(x_1,h)\pi$.
\end{proof}


\begin{lem}
Let $h,g\in H$.  Then $A_{h}=[\emptyset]_{hg^{-1}}A_{g}$.
\end{lem}

\begin{proof}
Let $\chi=x_1x_2x_3\ldots\in X^{\omega}$ and let $g,h\in H$. The permutation $hg^{-1}h_{x\cdot hg^{-1}}$ is an 
element of $H$ and takes the letter $x$ into the orbit transversal $R$.  This implies that
\begin{eqnarray*}
hg^{-1}h_{x\cdot hg^{-1}}&=&h_x.
\end{eqnarray*}
Now consider $\chi\cdot [\emptyset]_{hg^{-1}}A_{g}$, in this case
\begin{eqnarray*}
\chi\cdot [\emptyset]_{hg^{-1}}A_{g}&=&(x_1x_2x_3...)\cdot [\emptyset]_{hg^{-1}}A_{g}\\
&=&(x_1\cdot hg^{-1}g)\| (x_2\cdot hg^{-1}h_{x_1\cdot hg^{-1}})\| (x_3\cdot hg^{-1}h_{x_2\cdot hg^{-1}})\ldots\\
&=&(x_1\cdot h) \| (x_2\cdot h_{x_1})\| (x_3\cdot h_{x_2})\ldots\\
&=&\chi\cdot A_h.
\end{eqnarray*}
\end{proof}

The following lemma now applies to calculations in $V_n(GH)$.
\begin{lem}
Let $x\in X$ (so that $h_x\in H$) and $g\in N_{S_n}(H)\cap\text{Stab}_{S_n}(R)$.  
Then $[\emptyset]_{g}A_{h_x}=A_{h_{x\cdot g^{-1}}}[\emptyset]_{g}$.
\end{lem}

\begin{proof}
Let $\chi=x_1x_2x_3\ldots\in X^{\omega}$, $x\in X$, and let $g\in N_{S_n}(H)\cap\text{Stab}_{S_n}(R)$. Recall that $h_x^{-1}gh_{x\cdot g}=g$
from Lemma \ref{lem:semiconjugacy}.  
Now consider $\chi\cdot [\emptyset]_{g}A_{h_x}$, we have
\begin{eqnarray*}
\chi\cdot [\emptyset]_{g}A_{h_x}&=&(x_1x_2x_3...)\cdot [\emptyset]_{g}A_{h_x}\\
&=&(x_1\cdot gh_x)\| (x_2\cdot gh_{x_1\cdot g})\| (x_3\cdot gh_{x_2\cdot g})\ldots\\
&=&(x_1\cdot h_{x\cdot g^{-1}}h_{x\cdot g^{-1}}^{-1}gh_x)\| 
(x_2\cdot h_{x_1}h_{x_1}^{-1}gh_{x_1\cdot g})\| (x_3\cdot h_{x_2}h_{x_2}^{-1}gh_{x_2\cdot g})\ldots\\
&=&(x_1\cdot h_{x\cdot g^{-1}}g)\| (x_2\cdot h_{x_1}g)\| (x_3\cdot h_{x_2}g)\ldots\\
&=&\chi\cdot A_{h_{x\cdot g^{-1}}} [\emptyset]_g.
\end{eqnarray*}
\end{proof}

\section{Isomorphism}\label{sec:Isomorphism}

The following theorem shows the isomorphisms between different Thompson-like groups $V_n(G)$ for a fixed $n$
via conjugation by $A_{id}:=A_{H,R,id}$ constructed for a specific semiregular group $H$ and choice of transversal $R$.

\setcounter{thm}{1}
\begin{thm}\label{thm:isomorphism}
Let $H\leq S_n$ be semiregular, $R$ be an orbit transversal of $H$, and $A_{id}$ be the usual transducer.
Then for all $G\leq N_{S_n}(H)\cap\text{Stab}_{S_n}(R)$, the set $HG$ is a group and
the mapping $\phi:V_n(HG)\to V_n(G)$ 
defined by $(v)\phi=A_{id}^{-1}vA_{id}$ is an isomorphism. 
\end{thm}

\begin{proof}
This proof is split into distinct parts describing where the generators of $V_n(HG)$ and $V_n(G)$ are taken
under conjugation by $A_{id}$ and $A_{id}^{-1}$ respectively. We begin by describing a generating set for $V_n$,
which when combined with iterated permutations will generate $V_n(HG)$ and $V_n(G)$. The generators of $V_n$ we use
are called \emph{small swaps} and are defined to be those elements of $V_n$ that `swap' two incomparable cones, i.e $(\omega_1\;\omega_2)\in V_n$,
such that $[\omega_1]\cup[\omega_2]\neq \CS_n$, and is therefore `small'. See \cite{Gens_of_V_n} for more details.
\\
\\
Small Swaps under Conjugation:\\

Let $\rho_1$ and $\rho_2$ be incomparable words in $X^*$ such that $v=(\rho_1,\rho_2)$ is a small swap.
Recall that
\[
\text{supp}\big((v)\phi\big)=\big(\text{supp}(v)\big)\cdot A_{id}=\big\{[\rho_i\cdot A_{id}]\big|i=1,2\big\}
\]
Let $(\rho_1\cdot A_{id})\|\chi\in X^{\omega}$ be a word in supp$\big((v)\phi\big)$. Then
\begin{eqnarray*}
\big((\rho_1\cdot A_{id})\|\chi\big)\;\cdot (v)\phi &=& (\rho_1\cdot A_{id}\|\chi)\cdot A_{id}^{-1}vA_{id}\\
&=& (\rho_2\cdot A_{id})\|(\chi\cdot A_{(\rho_1,id)\pi}^{-1}A_{(\rho_2,id)\pi})\\
&=& (\rho_2\cdot A_{id})\|(\chi\cdot A_{(\rho_1,id)\pi}^{-1}A_{(\rho_1,id)\pi}\lfloor\emptyset\rfloor_{(\rho_1,id)\pi^{-1}(\rho_2,id)\pi})\\
&=& (\rho_2\cdot A_{id})\|(\chi\cdot \lfloor\emptyset\rfloor_{(\rho_1,id)\pi^{-1}(\rho_2,id)\pi}).
\end{eqnarray*}
This shows that $(v)\phi$ acts as $(\rho_1\cdot A_{id}\;\rho_2\cdot A_{id})\lfloor\rho_2\cdot A_{id}\rfloor_{(\rho_1,id)\pi^{-1}(\rho_2,id)\pi}$
on words with prefix $(\rho_1)A_{id}$.
The case for $\rho_2$ is similar giving
\[(v)\phi=(\rho_1\cdot A_{id}\;\rho_2\cdot A_{id})\lfloor\rho_2\cdot A_{id}\rfloor_{(\rho_1,id)\pi^{-1}(\rho_2,id)\pi}
\lfloor\rho_1\cdot A_{id}\rfloor_{(\rho_2,id)\pi^{-1}(\rho_1,id)\pi}\in V_n.\]
\\
Iterated Permutations under Conjugation:\\

Let $s\in HG$ and $\rho\in\{1,2,...,n\}^*$.  This implies that is a unique $h\in H$ and $g\in G$ such that $s=hg$ and
therefore $[\emptyset]_s=[\emptyset]_h[\emptyset]_g$
Now consider $\big([\rho]_s\big)\phi$.  Again,

\[
\text{supp}\big(([\rho]_s)\phi\big)=\big(\text{supp}([\rho]_s)\big)\cdot A_{id}=[\rho\cdot A_{id}].
\]

Then for a word $(\rho\cdot A_{id})\|\chi\in\supp(([\rho]_s)\phi)$, we have

\begin{align*}
((\rho\cdot A_{id})\|\chi)\cdot \big([\rho]_s\big)\phi=&\big((\rho\cdot A_{id})\|\chi\big)\cdot A_{id}^{-1}[\rho]_s A_{id}\\
=&(\rho\cdot A_{id})\|(\chi\cdot A_{(\rho,id)\pi}^{-1}[\emptyset]_s A_{(\rho,id)\pi})\\
=&(\rho\cdot A_{id})\|(\chi\cdot A_{(\rho,id)\pi}^{-1}[\emptyset]_h[\emptyset]_g A_{(\rho,id)\pi})\\
=&(\rho\cdot A_{id})\|(\chi\cdot A_{(\rho,id)\pi}^{-1}[\emptyset]_h A_{(\rho\cdot g^{-1},id)\pi}[\emptyset]_g)\\
=&(\rho\cdot A_{id})\|(\chi\cdot A_{(\rho,id)\pi}^{-1}A_{h(\rho\cdot g^{-1},id)\pi}[\emptyset]_g)\\
=&(\rho\cdot A_{id})\|(\chi\cdot \lfloor\emptyset\rfloor_{(\rho\cdot,id)\pi^{-1}h(\rho\cdot g^{-1},id)\pi}[\emptyset]_g).\\
\end{align*}

Summing up, we get that
\begin{eqnarray*}
([\rho]_t)\phi&=&\lfloor\rho\cdot A_{id}\rfloor_{(\rho\cdot,id)\pi^{-1}h(\rho\cdot g^{-1},id)\pi}[\rho\cdot A_{id}]_g\in V_n(G).
\end{eqnarray*}

We have now shown that $\phi$ maps $V_n(G)$ into $V_n$ by generators.
To show that $\phi$ is onto, we consider the preimage of a generating set of $V_n(G)$.\\
\\
Small Swaps under Inverse Conjugation:

Let $\rho_1\in X^*$ and $\rho_2\in X^*$ such 
that $v=(\rho_1\;\rho_2)$ is a small swap.  Consider $(v)\phi^{-1}$,

\[
\supp\big((v)\phi^{-1}\big)=\supp(v)\cdot A^{-1}_{id}=\big\{[\rho_i\cdot A^{-1}_{id}]\big| i=1,2 \big\}
\]

and consider the action of $(v)\phi^{-1}$ on the infinite word $\rho_1\cdot A_{id}^{-1}\|\chi$:

\begin{eqnarray*}
(\rho_1\cdot A^{-1}_{id}\|\chi)\cdot(v)\phi&=&(\rho_1\cdot A^{-1}_{id}\|\chi)\cdot
A_{id}(\rho_1\;\rho_2)A^{-1}_{id}\\
&=&(\rho_2\|\chi\cdot A_{(\rho_1,id)\pi'})\cdot A_{id}^{-1}\\
&=&(\rho_2\cdot A^{-1}_{id})\|(\chi\cdot A_{(\rho_1,id)\pi'}A^{-1}_{(\rho_2,id)\pi'})\\
&=&(\rho_2\cdot A^{-1}_{id})\|(\chi\cdot [\emptyset]_{(\rho_1,id)\pi'\big((\rho_2,id)\pi'\big)^{-1}}
A_{(\rho_2,id)\pi'}A^{-1}_{(\rho_2,id)\pi'})\\
&=&(\rho_2\cdot A^{-1}_{id})\|(\chi\cdot [\emptyset]_{(\rho_1,id)\pi'\big((\rho_2,id)\pi'\big)^{-1}}).
\end{eqnarray*}
The case for the cone $[\rho_2\cdot A^{-1}_{id}]$ is similar which gives
\[
(v)\phi^{-1}=\big(\rho_1\cdot A^{-1}_{id}\;\rho_2\cdot A^{-1}_{id}\big)
[\rho_1\cdot A^{-1}_{id}]_{(\rho_2,id)\pi'\big((\rho_1,id)\pi'\big)^{-1}}
[\rho_2\cdot A^{-1}_{id}]_{(\rho_1,id)\pi'\big((\rho_2,id)\pi'\big)^{-1}}\in V_n(H)\leq V_n(GH).
\]
\\
Iterated Permutations under Inverse Conjugation:\\

The last generators of $V_n(G)$ are the iterated permutations $[\rho]_g$ where $\rho\in X^*$
and $g\in G$.  By examining the previous calculations on iterated permutations, we see
\begin{eqnarray*}
([\rho\cdot A^{-1}_{id}]_g)\phi&=&\lfloor\rho\rfloor_{(\rho\cdot A^{-1}_{id},id)\pi^{-1}g(\rho\cdot A^{-1}_{id},id)\pi g^{-1}}
[\rho]_g.\\
\end{eqnarray*}
Since $\lfloor\rho\rfloor_{(\rho\cdot A^{-1}_{id},id)\pi^{-1}g(\rho\cdot A^{-1}_{id},id)\pi g^{-1}}$ is in $V_n$
and can therfore be expressed as a product of small swaps, applying $\phi^{-1}$ gives
\[([\rho]_g)\phi^{-1}=[\rho\cdot A^{-1}_{id}]_g
\big(\lfloor\rho\rfloor_{(\rho\cdot A^{-1}_{id},id)\pi^{-1}g(\rho\cdot A^{-1}_{id},id)\pi g^{-1}}^{-1}\big)\phi^{-1}\in V_n(GH).\]

All together, $\phi$ has been shown to be an onto mapping from $V_n(HG)$ to $V_n(G)$ and
since the mapping is via conjugation, it is also one-to-one and a homomorphism.  This
demonstrates that $\phi$ is an isomorphism as desired.
\end{proof}

Of particular interest is when the group $G\leq N_{S_n}(H)\cap \text{Stab}_{S_n}(R)$ is the trivial group.
This says
\[V_n(H)\cong V_n\]
and we can state the following Theorem using Corollary \ref{cor:semiregularnecessity} and Theorem \ref{thm:isomorphism} 

\setcounter{thm}{0}
\begin{thm}\label{cor:iffsemi}
  Let $G \leq S_n$. Then $V_n(G) \cong V_n$ if and only if $G$ is semiregular. 
\end{thm}

\section{Examples and Remarks}
Using Theorem \ref{thm:isomorphism}, as well as some extensions of the non-isomorphism results in Section 2,
we are able to distinguish
several isomorphism classes for small $n$ and a few are described here.\\
\\
$n$=2:\\

\newcommand{\Ttwo}{\mathcal{T}_2}

The symmetric group on two points has two subgroups, the trivial group and itself.  Adding the action of the trivial group to $V_2$ merely
gives the familiar Thompson's group $V$, whereas $V_2(S_2)$ has the extra action of elements of the form $[\rho]_{(1\;2)}$.  One way to picture
the action of these elements on $\Ttwo$ is `reflecting' or `flipping' the entire tree beneath the node $\rho$.  
However, the symmetric group $S_2$ is semiregular,
meaning $V_2\cong V_2(S_2)$.  The transducer $A_{id}$ used for this isomorphism is shown in Figure \ref{fig:twostate}. We provide
two descriptions of this isomorphism, an informal heuristic argument and a formal description.

Following the action of the transducer down the infinite tree $\Ttwo$, we see that $A_{id}$ enters the state $id$ after travelling down the left
branch (corresponding to 1) and enters the state $(1\;2)$ down the right branch (corresponding to 2).  Using this picture, $A_{id}$ will
`flip' the right half of each node in the tree.  Conjugating by this will undo the infinite action of $[\rho]_{(1\;2)}$ at node $\rho$,
cutting it off after one level.  The uniform nature of $A_{id}$ will cancel with itself when conjugating elements of $V$, adding only a few 
additional small swaps at a finite number of locations. 

From the calculations
in Section \ref{sec:Isomorphism}, we can show the action of the isomorphism explicitly.
Each element $v\in V_2$ can be described using two partitions of $\CS_n$ into cones, 
$\big\{[\alpha_i]\big\}_{i=1}^k$ and $\big\{[\beta_i]\big\}_{i=1}^k$, such that 
$(\alpha_i\|\chi)\cdot v= \beta_i\|\chi$ for all $\chi\in X^\omega$.  This description extends nicely to $w\in V_2(S_2)$, where each
element requires a third list $\{g_i\}_{i=1}^k$ of permutations in $S_2$, describing the action of $w$ on the infinite tail.  The action of $w$
can be summarized in the notation earlier in this paper as $w=v\prod_{i=1}^k[\beta_i]_{g_i}$. Note that $g_i$ may be the trivial permutation, in which
case $[\beta_i]_{g_i}$ is the identity.  On the other hand, referring back to Section \ref{sec:Isomorphism}, it is clear that 
\begin{eqnarray*}
\big([\beta_i]_{(1\;2)}\big)\phi&=&\lfloor\beta_i\cdot A_{id}\rfloor_{(1\;2)}.
\end{eqnarray*}

The conjugate $(v)\phi$ is harder to describe, but it is derived in nearly the same fashion as small swaps.  If $\alpha_i$ and $\beta_i$
end in the same letter, i.e. both end in 1 or 2, then $(\alpha_i,id)\pi=(\beta_i,id)\pi$ and in particular,
$(\alpha_i,id)\pi^{-1}(\beta_i,id)\pi=id$. This implies that 
\begin{eqnarray*}
\big((\alpha_i\cdot A_{id})\|\chi\big)\cdot(v)\phi&=&(\beta_i\cdot A_{id})\|\chi.
\end{eqnarray*}
However, when $\alpha_i$ and $\beta_i$
end in different letters, $(\alpha_i,id)\pi^{-1}(\beta_i,id)\pi=(1\;2)$ since there are only two elements in $S_2$. In this case,
\begin{eqnarray*}
\big((\alpha_i\cdot A_{id})\|\chi\big)\cdot(v)\phi&=&(\beta_i\cdot A_{id})\|(\chi\cdot \lfloor\emptyset\rfloor_{(1\;2)}).
\end{eqnarray*}
Let $I=\{i|\alpha_i\text{ and }\beta_i\text{ end in different letters}\}$.  Then the element
$(w)\phi$ is a product of three elements in $V$: one corresponding to the partitions $\{[\alpha_i\cdot A_{id}]\}_{i=1}^k$ and
$\{[\beta_i\cdot A_{id}]\}_{i=1}^k$, the element $\prod_{i\in I} \lfloor\beta_i\cdot A_{id}\rfloor_{(1\;2)}$, and the element
$\prod_{i=1}^k\lfloor\beta_i\cdot A_{id}\rfloor_{g_i}$.  This can be generalized to higher $n$ and other semiregular groups, but
quickly becomes cumbersome and perhaps unhelpful for studying these groups in generality.\\
\\
$n$=3:\\

\setcounter{thm}{10}
For permutations on three points, there are four unique subgroups of $S_3$ up to conjugation: 
the trivial group, $S_2:=\langle(1\;2)\rangle$, $C_3:=\langle(1\;2\;3)\rangle$,
and $S_3$.  It is clear that $C_3$ is semiregular, so $V_3\cong V_3(C_3)$, but both $S_2$ and $S_3$ are not.  However, the normalizer
of $C_3$ is all of $S_3$, and $S_2$ stabilizes a potential orbit transversal of $C_3$, namely the set $\{3\}$. This fits the criteria for
Theorem \ref{thm:isomorphism} and therefore $V_3(S_3)=V_3(C_3S_2)\cong V_3(S_2)$.  This splits $V_3(G)$ into two distinct isomorphism classes.
Both isomorphisms are built using the same transducer in Figure \ref{fig:triangular}, perhaps with a different orbit transversal $R$, which can
be done simply with a relabeling of the alphabet $\{1,2,3\}$. 
Note that comparing this to the case when $n=2$, this highlights the subtle yet rather intuitive fact that the number of 
points a permutation group acts on will significantly change orbit dynamics.
\begin{rem}
$V_n(G)\cong V_n(H)$ does not imply that $V_m(G) \cong V_m(H)$, for $m\neq n$.
\end{rem}

\noindent $n$=4:\\

A similar process can used to distinguish the four isomorphism classes for $V_4(G)$.
\[V_4\cong V_4\Big(\big\langle(1\;2)(3\;4)\big\rangle\Big)\cong
V_4\Big(\big\langle(1\;2\;3\;4)\big\rangle\Big)\cong V_4\Big(\big\langle(1\;2)(3\;4),(1\;3)(2\;4)\big\rangle\Big)\]
\[V_4\Big(\big\langle(1\;2)\big\rangle\Big)\cong V_4\Big(\big\langle(1\;2),(1\;2)(3\;4)\big\rangle\Big)\cong
V_4\Big(\big\langle(1\;2),(1\;3\;2\;4)\big\rangle\Big)\]
\[V_4\Big(\big\langle(1\;2\;3)\big\rangle\Big)\cong V_4(A_4)\]
\[V_4(S_3)\cong V_4(S_4)\]
The method of examining orbit structure in Lemma \ref{lem:anticonjugacy} can be used to differentiate these classes,
particularly the orbits near fixed points of elements like $[\emptyset]_g$.  Farley and Hughes's non-isomorphism result
in \cite{FarleyHughesFSS} could also be used to distinguish these classes.
This example demonstrates the importance of the orbit
structure of the permutation group used in the extension over than the isomorphism type of the group used.
\begin{rem}
There are integers $n>2$ and $G\cong H$ subgroups of $S_n$ with $V_n(G)\not \cong V_n(H)$.
\end{rem}

Thus far, we have discussed whether a given group is isomorphic to $V_n$, but there is also
an interest in understanding when a group can be found as a subgroup of $V_n$.  One cause for such interest comes
from a conjecture of Lehnert, modified by Bleak, Matucci, and Neuh\"{o}ffer in \cite{BMNcoCF}, that Thompson's group $V$ is a universal group with context-free co-word problem, i.e. a universal $co\mathscr{CF}$ group (so every finitely generated subgroup of $V$ is a $co\mathscr{CF}$ group, and all $co\mathscr{CF}$ groups embed into $V$).
Related to this, Farley in 
\cite{FarleyCounters} describes a family of groups which he proves are all $co\mathscr{CF}$ groups, which family includes the groups $V_n(G)$ discussed here.  Farley proposes that some of these groups might be used to provide counterexamples to the Lehnert conjecture, which would occur if one can prove that one of these groups cannot embed into $V$.  The 
techniques of this paper might be useful in finding transducers that could be used to embed such groups into 
$V_n$ via topological conjugation.
This is an area for future investigations.

\bibliographystyle{plain}
\bibliography{references}
\end{document}